\documentclass[letterpaper,11pt,leqno]{amsart}
\usepackage[T1]{fontenc}
\usepackage[utf8]{inputenc}
\usepackage{lmodern}
\usepackage[english]{babel}
\usepackage{amsmath,a4wide}
\usepackage{xfrac}
\usepackage{esint}
\usepackage{stmaryrd,mathrsfs,bm,amsthm,mathtools,yfonts,amssymb,color,braket,booktabs,graphicx,graphics,amsfonts}
\usepackage{latexsym,microtype,indentfirst,hyperref}
\usepackage{xcolor}
\usepackage{courier}

\newtheorem{theorem}{Theorem}[section]

\newtheorem{lemma}[theorem]{Lemma}

\theoremstyle{definition}
\newtheorem{definition}[theorem]{Definition}

\newtheorem{assumption}[theorem]{Assumption}

\numberwithin{equation}{section}
\numberwithin{subsection}{section}

\newcommand{\spt}{\mathrm{spt}} %support (usually of a measure)
\newcommand{\dist}{\mathrm{dist}} %distance
\newcommand{\tr}{\mathrm{tr}} %trace
\newcommand{\di}{\mathrm{div}} %divergence
\newcommand{\weakstar}{\stackrel{*}{\rightharpoonup}}

\title[Level set mean curvature flow with Neumann boundary conditions]{Level set mean curvature flow \\ with Neumann boundary conditions}

\author[S. Aimi]{Satoru Aimi}
\address[S.Aimi]{Department of Mathematics, Tokyo Institute of Technology, 2-12-1 Ookayama, Meguro-ku, Tokyo 152-8551, Japan}
\email{aimisatoru@gmail.com}

\thanks{Department of Mathematics, Tokyo Institute of Technology, 2-12-1 Ookayama, Meguro-ku, Tokyo 152-8551, Japan}
\thanks{E-mail: aimisatoru@gmail.com}

\makeatletter
\@namedef{subjclassname@2020}{%
  \textup{2020} Mathematics Subject Classification}
\makeatother
\subjclass[2020]{53E10, 49Q20}

\begin{document}

\begin{abstract}
We investigate the relation between the level set approach and the varifold approach for the mean curvature flow with Neumann boundary conditions. With an appropriate initial data, we prove that the almost all level sets of the unique viscosity level set solution satisfy Brakke's inequality  and a generalized Neumann boundary condition.\\

\textit{Keywords:} mean curvature flow, Neumann boundary conditions, level set method, varifolds.\\

\end{abstract}
\maketitle

%\tableofcontents

\section{Introduction}
Let $\Omega$ be a bounded domain with a smooth boundary $\partial\Omega$ in ${\mathbb R}^N$. We consider a family of hypersurfaces $\{\Gamma_t\}_{t\geq 0}$ in $\overline{\Omega}$ which evolves according to its mean curvature with Neumann boundary conditions, i.e.
\begin{equation} \label{MCF}
\begin{cases}
V=-H&\textrm{on}~\Gamma_t,\\
\Gamma_t\perp\partial \Omega,
\end{cases}
\end{equation}
where $V$ is the normal velocity of $\Gamma_t$ and $H$ is the mean curvature of $\Gamma_t$. In the case of $\Omega={\mathbb R}^N$, i.e. without boundary, the problem has long been studied and several different approaches have been proposed. One of these approaches is called the level set method (see \cite{CGG,Evans1}), another is based on geometric measure theory called Brakke flow (see \cite{Brakke}). 
The consistency between these approaches is studied by numerous researchers. In particular, we mention \cite{Evans4}, which showed a relation between the level set method and the Brakke flow when $\Omega={\mathbb R}^N$, namely, with an appropriate initial data, they showed that almost all level sets of the unique viscosity 
solution are Brakke flows.
In the present paper, we investigate the similar relation when the problem is supplemented by the 
Neumann boundary condition.\par

Here, we mention some works related to the mean curvature flow with boundary conditions.
Sato \cite{Sato} proved the existence of a unique viscosity level set solution of \eqref{MCF} and Giga and Sato \cite{GigaSato} established the comparison principle. In the classical setting, Stahl \cite{Stahl1} proved a short-time existence of a unique smooth mean curvature flow with Neumann boundary conditions. For a long-time existence result, 
Mizuno and Tonegawa \cite{MizuTone} introduced a notion of the Neumann boundary condition for Brakke flow and proved its existence via the Allen-Cahn equation (see also \cite{Kagaya}). In \cite{Edelen}, Edelen also studied 
the Brakke flow with Neumann boundary conditions and established its existence and a short-time regularity 
via the elliptic regularization approach.\par

In this paper, we investigate a connection between the level set solution of \eqref{MCF} and Brakke flow. First, we consider the following PDE:
\begin{equation} \label{intro:level set MCF}
\begin{cases}
\partial_t u=|\nabla u|\,\textrm{div}\left(\frac{\nabla u}{|\nabla u|}\right)&\textrm{in}~\Omega\times (0,\infty),\\
\nabla u\cdot \nu_{\partial \Omega}=0&\textrm{on}~\partial \Omega\times (0,\infty),\\
u=g&\textrm{on}~\Omega\times \{t=0\},
\end{cases}
\end{equation}
where $\nu_{\partial \Omega}$ denotes the outer unit normal of $\partial \Omega$, and $g:{\mathbb R}^N\rightarrow {\mathbb R}$ satisfies
\begin{equation}
\begin{cases}
\Gamma_0=\{x\in \overline{\Omega}~|~g(x)=0\},\\
\nabla g\cdot \nu_{\partial \Omega}=0\quad\textrm{on}~\partial \Omega.
\end{cases}
\end{equation}

Heuristically, this PDE asserts that each level set of $u$ evolves by \eqref{MCF}, if $u$ is smooth and $|\nabla u|\not=0$. Mostly following the similar argument in \cite{Evans4}, we show that almost all level sets of $u$ are Brakke flows if $g$ is chosen appropriately. The main interest of the present paper lies in the characterization 
of the boundary condition that the Brakke flow satisfies. It turned out that the first variation of 
the Brakke flow on $\partial\Omega$ is bounded and 
perpendicular to $\partial\Omega$ almost everywhere, which was not guaranteed in \cite{MizuTone}.
Moreover, the class of admissible test
functions for Brakke's inequality is strictly larger than those in \cite{Edelen,MizuTone} in that our 
test functions $\phi$ do not need to satisfy the Neumann boundary conditions $\nabla\phi\cdot\nu_{\partial\Omega}=0$ which
was required in \cite{Edelen,MizuTone}. One intuitive reason for the difference is that the conclusion here
is for almost all level sets and they
have null $(N-1)$-dimensional measure on $\partial\Omega$, while the Brakke flows in \cite{Edelen,MizuTone} may have
non-trivial measure on $\partial\Omega$. 

Since almost all level sets are in addition unit-density, they are smooth almost everywhere by the general regularity theory in \cite{Brakke,KasaTone,Tone}. The regularity up to the boundary is not known so far on the other hand.

The organization of the paper is as follows. In Section 2, we state the basic notation and main results. In Section 3, we obtain a crucial $L^1$ estimate of the curvature which is a localized version of \cite{Evans4}. This estimate leads to various convergence results, allowing the derivations of the first variation formula and Brakke's inequality for the level set solution. In Section 4, we interpret formulae obtained in Section 3 into varifolds settings. We show that the first variation is bounded and 
conclude that almost all level sets of $u$ are Brakke flows with Neumann boundary conditions. In Appendix, a technical construction of the initial data $g$ is carried out. 

\section{Preliminaries and main results}
%We describe basic notation and main results in this paper.
\subsection{Notation and assumption}
Let $N\geq 2$ be an integer. For a $C^1$ vector field $X$ in ${\mathbb R}^N$, we write $\nabla X:=(\partial_i X^j)_{i,j}$. The symbol $\mu\lfloor_{A}$ denotes the restriction of a measure $\mu$ to a set $A$. Throughout the paper, we make the following assumption on $\Omega$ and initial surface $\Gamma_0$.

\begin{assumption}\label{asm}
Assume that $U,~\Omega\subseteq {\mathbb R}^N$ with $U\cap \Omega\neq \emptyset$ are bounded open sets with smooth boundaries and we set the initial surface $\Gamma_0:=\partial U\cap \overline{\Omega}$. Assume that $\partial U\cap \partial \Omega\not=\emptyset$ and that $\Gamma_0$ meets $\partial \Omega$ orthogonally.
\end{assumption}

We write $\nu_{\partial \Omega}$ for the outer unit normal of $\partial \Omega$. 

\subsection{Varifolds}
A set $\textrm{G}_k({\mathbb R}^N)$ denotes the Grassmannian, i.e. $\textrm{G}_k({\mathbb R}^N)$ is a collection of $k$-dimensional linear subspaces of ${\mathbb R}^N$. We often identify $S\in \textrm{G}_k({\mathbb R}^N)$ with the orthogonal projection from ${\mathbb R}^N$ onto $S$. A $k$-varifold in $\overline{\Omega}$ is a Radon measure on $\overline{\Omega}\times \textrm{G}_k({\mathbb R}^N)$. By the Riesz representation theorem, $k$-varifolds in $\overline{\Omega}$ correspond bijectively to bounded linear functionals on the space of continuous functions $C(\overline{\Omega}\times \textrm{G}_k({\mathbb R}^N))$. For a $k$-varifold $V$ in $\overline{\Omega}$, 
$\|V\|$ denotes its weight measure, i.e.
\begin{equation}
\|V\|(\phi):=\int_{\overline{\Omega}\times \textrm{G}_k({\mathbb R}^N)}\phi(x)\,dV(x,S)
\quad\textrm{for}~\phi\in C(\overline{\Omega}).
\end{equation}
We write $\delta V$ for the first variation of $V$, i.e.
\begin{equation}
\delta V(X):=\int_{\overline{\Omega}\times \textrm{G}_k({\mathbb R}^N)}\textrm{div}_S(X)\,dV(x,S)
\quad\textrm{for}~X\in C^1(\overline{\Omega};{\mathbb R}^N),
\end{equation}
where $\textrm{div}_S(X):=\tr(S\nabla X)$. We write $\|\delta V\|(\overline{\Omega})$ for the total variation of $\delta V$ in $\overline{\Omega}$, i.e.
\begin{equation}
\|\delta V\|(\overline{\Omega}):=\sup\left\{\delta V(X)~\left|~X\in C^1(\overline{\Omega};{\mathbb R}^N),~\|X\|_{\infty}\leq 1\right\}\right..
\end{equation}
A $k$-rectifiable set $\Gamma\subseteq \overline{\Omega}$ induces a $k$-varifold in $\overline{\Omega}$ as in
\begin{equation}
V_{\Gamma}(\phi):=\int_{\Gamma}\phi(x,T_x\Gamma)\,d\mathcal{H}^k(x)
\quad\textrm{for}~\phi\in C(\overline{\Omega}\times \textrm{G}_k({\mathbb R}^N)).
\end{equation}
Here, note that a $k$-rectifiable set $\Gamma$ has a unique approximate tangent space $T_x\Gamma$ for $\mathcal{H}^k$-a.e. $x\in \Gamma$. A varifold induced from a rectifiable set as above is said to be unit-density. We work mainly with unit-density varifolds.

\subsection{Level set solution}
Consider the following singular second order parabolic PDE:
\begin{equation} \label{level set MCF}
\begin{cases}
\partial_t u=|\nabla u|\,\textrm{div}\left(\frac{\nabla u}{|\nabla u|}\right)&\textrm{in}~\Omega\times (0,\infty),\\
\nabla u\cdot \nu_{\partial \Omega}=0&\textrm{on}~\partial \Omega\times (0,\infty),\\
u=g&\textrm{on}~\Omega\times \{t=0\}.
\end{cases}
\end{equation}
We approximate \eqref{level set MCF} by the following nonsingular PDE for $\varepsilon\in (0,1)$:
\begin{equation} \label{approx eq} \tag{\ref{level set MCF}$_{\varepsilon}$}
\begin{cases}
\partial_t u^{\varepsilon}=\sqrt{|\nabla u^{\varepsilon}|^2+\varepsilon^2}\,\textrm{div}\left(\frac{\nabla u^{\varepsilon}}{\sqrt{|\nabla u^{\varepsilon}|^2+\varepsilon^2}}\right)&\textrm{in}~\Omega\times (0,\infty),\\
\nabla u^{\varepsilon}\cdot \nu_{\partial \Omega}=0&\textrm{on}~\partial \Omega\times (0,\infty),\\
u^{\varepsilon}=g&\textrm{on}~\Omega\times \{t=0\}.
\end{cases}
\end{equation}
We construct a suitable initial data $g$ with $\{g=0\}=\Gamma_0$ which is smooth and which has a compatible 
Neumann condition but we defer the description of $g$ at this point. Here, we note that the equation \eqref{approx eq} is uniformly parabolic, provided $\|\nabla u^{\varepsilon}\|_{\infty}$ is finite. Thus the existence of a unique bounded solution $u^{\varepsilon}\in C^{\infty}(\overline{\Omega}\times[0,\infty))$ follow by an apriori estimate for $\|\nabla u^{\varepsilon}\|_{\infty}$ and the theory of parabolic equations (see for instance \cite{Huisken, Lieberman}).\par
Furthermore, one may have
\begin{equation}\label{bounds for level set solution}
\sup_{\substack{x\in \Omega\\0\leq t\leq T}}|u^{\varepsilon}|,~|\nabla u^{\varepsilon}|,~|\partial_t u^{\varepsilon}|\leq C_T
\end{equation}
with a constant $C_T>0$ independent of $\varepsilon$. Indeed, like in \cite{Evans1}, the bounds for $u^{\varepsilon}$ and $\partial_t u^{\varepsilon}$ follow by the maximum principle with Neumann conditions. Set $W:=|\nabla u^{\varepsilon}|^2$ and $\sigma_{ij}(p):=\delta_{ij}-\frac{p_ip_j}{|p|^2+\varepsilon^2}$, and let $d:{\mathbb R}^N\rightarrow [-1,1]$ be a smooth function such that $\nabla d\cdot \nu_{\partial\Omega}\geq 1$ on $\partial \Omega$. Then we can see
\[\partial_t W=\sigma_{ij}(\nabla u^{\varepsilon})\partial_{ij}W+\partial_k\sigma_{ij}(\nabla u^{\varepsilon})\partial_{ij}u^{\varepsilon}\partial_k W-2\sigma_{ij}(\nabla u^{\varepsilon})\partial_{ik}u^{\varepsilon}\partial_{jk}u^{\varepsilon}\leq \sigma_{ij}(\nabla u^{\varepsilon})\partial_{ij}W,\]
where we use the following computation and the positive-definiteness of $\sigma_{ij}$:
\begin{align*}
&\partial_k \sigma_{ij}(\nabla u^{\varepsilon})\partial_{ij}u^{\varepsilon}\partial_k W=\left(\frac{2\partial_i u\partial_j u \partial_k u}{(|\nabla u^{\varepsilon}|^2+\varepsilon^2)^2}-\frac{\delta_{ik}\partial_j u+\delta_{jk}\partial_i u}{|\nabla u^{\varepsilon}|^2+\varepsilon^2}\right)\partial_{ij}u^{\varepsilon}\partial_k W\\
&\quad =\frac{|\nabla u\cdot \nabla W|^2}{(W+\varepsilon^2)^2}-\frac{|\nabla W|^2}{W+\varepsilon^2}\leq 0.
\end{align*}
On the other hand, we calculate $\nabla W\cdot \nu_{\partial\Omega}\leq C\|\nabla \nu_{\partial\Omega}\|_{\infty}W$ on $\partial \Omega$ for some constant $C$ (see \cite[Proposition 2.1]{GigaOhmuraSato}). Therefore, for sufficiently large $\alpha,\beta>0$, $v:=e^{-\alpha t-\beta d}W$ is a subsolution to a linear 2nd order parabolic PDE with a non-positive 0th order coefficient. Hence we conclude by the maximum principle,
\[|\nabla u^{\varepsilon}(x,t)|^2\leq e^{\alpha T+2\beta}\sup_{x\in \Omega}|\nabla g|^2\]
for $x\in \Omega, 0\leq t\leq T$.\par
By the general theory of viscosity solutions, there exists a unique viscosity solution $u$ to \eqref{level set MCF}, and as $\varepsilon\to 0$, $u^{\varepsilon}$ converges to $u$ locally uniformly on $\overline{\Omega}\times [0,\infty)$. 
See \cite{Giga,GigaSato,Sato} for the existence and uniqueness of the viscosity solution in this case. 
Due to the degeneracy of \eqref{level set MCF}, $u$ is expected to be merely Lipschitz continuous in time and space in general. 
\begin{definition}
For the unique viscosity solution $u$ of \eqref{level set MCF}, we set
\begin{equation}
\Gamma_t^{\gamma}:=\{x\in \overline{\Omega}~|~u(x,t)=\gamma\}
\end{equation}
for $(\gamma,t)\in {\mathbb R}\times [0,\infty)$, and define (note that $u$ is differentiable a.e. due to
Rademacher's theorem)
\begin{equation}
\nu:=
\begin{cases}
\frac{\nabla u}{|\nabla u|}&\textrm{if}~|\nabla u|>0,\\
0&\textrm{if}~|\nabla u|=0,\\
\end{cases}
\quad
H:=
\begin{cases}
\frac{\partial_t u}{|\nabla u|}&\textrm{if}~|\nabla u|>0,\\
0&\textrm{if}~|\nabla u|=0.
\end{cases}
\end{equation}
By \cite[Lemma 6.1]{Evans4}, $\Gamma_t^{\gamma}$ is $(N-1)$-rectifiable for a.e. $(\gamma,t)\in\mathbb R\times
(0,\infty)$ and $\Gamma_t^{\gamma}$ induces a unit-density $(N-1)$-varifold $V_t^{\gamma}$ for such $(\gamma,t)$.
\end{definition}

\subsection{Main results}
Now we describe the main results in the present paper.
We need to have a suitable initial data with the following property.
\begin{lemma} \label{initial datum}
There exists a smooth and bounded function $g:\overline{\Omega}\rightarrow {\mathbb R}$ such that
\begin{itemize}
\item[(i)] $\Gamma_0=\{x\in \overline{\Omega}~|~g(x)=0\}$, 
\item[(ii)] $\nabla g\cdot \nu_{\partial \Omega}=0$ on $\partial \Omega$,
\item[(iii)] $\displaystyle \sup_{0<\varepsilon<1}\int_{\Omega}\left|\di\left(\frac{\nabla g}{\sqrt{|\nabla g|^2+\varepsilon^2}}\right)\right|\,dx<\infty$.
\end{itemize}
\end{lemma}
With this $g$, we define all the relevant quantities (the proof of Lemma \ref{initial datum} is somewhat
technical and is deferred to Appendix). 
Then we have the following boundedness of the first variation of $V_t^{\gamma}$.
\begin{lemma} \label{main:total var fin}
For a.e. $(\gamma,t)\in\mathbb R\times(0,\infty)$, we have $\|\delta V_t^{\gamma}\|(\overline{\Omega})<\infty$
and for a.e. $\gamma\in\mathbb R$ and any $T>0$, we have
\begin{equation}\label{totalvar}
\int_0^T \|\delta V_t^{\gamma}\|(\overline{\Omega})\,dt<\infty.
\end{equation}
\end{lemma}

By Lemma \ref{main:total var fin}, we may define the tangential component of the first variation $\delta V_t^{\gamma}$ on $\partial \Omega$ as in \cite[Definition 2.4]{MizuTone}:

\begin{definition}\label{1st var on bdry}
For a.e. $(\gamma,t)\in\mathbb R\times(0,\infty)$ such that $\|\delta V_t^{\gamma}\|(\overline{\Omega})<\infty$, set
\begin{equation}
\delta V_t^{\gamma}\lfloor_{\partial \Omega}^{\top}(X):=\delta V_t^{\gamma}\lfloor_{\partial \Omega}(X-(X\cdot \nu_{\partial \Omega})\nu_{\partial \Omega})
\quad\textrm{for}~X\in C(\partial \Omega;{\mathbb R}^N).
\end{equation}
\end{definition}
%Then the Neumann boundary condition of $V_t^{\gamma}$ and the first variation formula are stated as follows:
With this definition, we have
\begin{theorem}\label{main:boundary conditions}
For a.e. $(\gamma,t)\in\mathbb R\times(0,\infty)$, we have
\begin{itemize}
\item[(i)] $\delta V_t^{\gamma}\lfloor_{\Omega}=H\nu \|V_t^{\gamma}\|$,
\item[(ii)] $H\in L^2(\|V_t^{\gamma}\|)$,
\item[(iii)] $\delta V_t^{\gamma}\lfloor_{\partial \Omega}^{\top}=0$.
\end{itemize}
\end{theorem}
The properties (i) and (ii) show that the first variation is absolutely continuous with respect to $\|V_t^\gamma\|$
inside $\Omega$ with $L^2$ generalized mean curvature, and (iii) shows that the first variation on $\partial\Omega$
is parallel to $\nu_{\partial\Omega}$ as a measure. The latter may be described that the $V_t^\gamma$ satisfies
the Neumann boundary condition in a generalized sense. 
Finally, $V_t^{\gamma}$ satisfies the following integral inequality called Brakke's inequality.

\begin{theorem}\label{main:Brakke ineq}
For a.e. $\gamma\in {\mathbb R}$, we have for a.e. $0\leq t_1<t_2<\infty$ and for any 
non-negative $\phi \in C^1(\overline{\Omega}\times [0,\infty))$ the inequality
\begin{equation}\label{bineq}
\int_{\overline{\Omega}}\phi(\cdot,t_2)\,d\|V_{t_2}^{\gamma}\|-\int_{\overline{\Omega}}\phi(\cdot,t_1)\,d\|V_{t_1}^{\gamma}\|\leq \int_{t_1}^{t_2}\int_{\overline{\Omega}}\partial_t\phi-H(\nu\cdot \nabla \phi)-\phi H^2\,d\|V_t^{\gamma}\|dt.
\end{equation}
\end{theorem}
As stated in Section 1, it is interesting to observe that we do not need to assume $\nabla\phi\cdot\nu_{\partial\Omega}=0$ on $\partial\Omega$, which was additionally required in \cite{Edelen,MizuTone}. We also note that for a.e. $\gamma$, $V_t^{\gamma}$ has the semi-decreasing property, namely, the mapping $t\mapsto \|V_t^{\gamma}\|(\phi)-C_{\phi}t$ is non-increasing for any non-negative $\phi\in C^2(\overline{\Omega})$. From this property, we may define $\|V_t^{\gamma}\|$ for all $t\in [0,\infty)$ so that $t\mapsto \|V_t^{\gamma}\|$ is left-continuous, and one may prove Theorem \ref{main:Brakke ineq} is still valid for all $0\leq t_1<t_2<\infty$.

\section{\texorpdfstring{$L^1$ estimate on the mean curvature and geometric formulae}{L1 estimate on the mean curvature and geometric formulae}}
We first show a crucial $L^1$ estimate for the approximate mean curvature $H^\varepsilon$. This estimate plays an important role in proceeding convergence arguments based on compensated compactness. In the latter part of Section 3, we establish the level set version of the first variation formula and Brakke's inequality. 

\begin{definition}
We define
\begin{equation}
\nu^{\varepsilon}:=\frac{\nabla u^{\varepsilon}}{\sqrt{|\nabla u^{\varepsilon}|^2+\varepsilon^2}},\quad
H^{\varepsilon}:=\textrm{div}(\nu^{\varepsilon}),\quad
A^{\varepsilon}:=I-\nu^{\varepsilon}\otimes \nu^{\varepsilon}.
\end{equation}
\end{definition}

\begin{theorem} \label{thm:L1 estimate}
We have
\begin{equation} \label{L1 estimate}
\sup_{\substack{0<\varepsilon<1\\t\geq 0}}\int_{\Omega}|H^{\varepsilon}(x,t)|\,dx<\infty.
\end{equation}
Moreover, for each $\varepsilon\in (0,1)$, the mapping $t\mapsto \int_{\Omega}|H^{\varepsilon}(x,t)|\,dx$ is non-increasing.
\end{theorem}

\begin{proof}
Let $\eta\in C^{\infty}({\mathbb R})$ be convex and $|\eta'|\leq 1$. For any $T>0$ and $t\in [0,T]$, we compute
\begin{equation}
\begin{aligned}
\frac{d}{dt}\int_{\Omega}\eta(H^{\varepsilon})\,dx&=\int_{\Omega}\eta'(H^{\varepsilon})\partial_t H^{\varepsilon}\,dx=\int_{\Omega}\eta'(H^{\varepsilon})\textrm{div}(\partial_t \nu^{\varepsilon})\,dx\\
&=\int_{\partial\Omega}\eta'(H^{\varepsilon})(\partial_t \nu^{\varepsilon}\cdot \nu_{\partial \Omega})\,d\mathcal{H}^{N-1}-\int_{\Omega}\eta''(H^{\varepsilon})\nabla H^{\varepsilon}\cdot \partial_t \nu^{\varepsilon}\,dx.
\end{aligned}
\end{equation}
From the boundary conditions of $u^{\varepsilon}$, we see $\partial_t \nu^{\varepsilon}\cdot \nu_{\partial \Omega}=\partial_t(\nu^{\varepsilon}\cdot \nu_{\partial \Omega})=0$ on $\partial \Omega$. On the other hand, since $\eta$ is convex and $A^{\varepsilon}$ is positive semi-definite,
\begin{equation}
\begin{aligned}
\eta''(H^{\varepsilon})\nabla H^{\varepsilon}\cdot \partial_t \nu^{\varepsilon}&=
\eta''(H^{\varepsilon})\nabla H^{\varepsilon}\cdot \left(A^{\varepsilon}\nabla H^{\varepsilon}+H^{\varepsilon}A^{\varepsilon}\frac{\nabla (\sqrt{|\nabla u^{\varepsilon}|^2+\varepsilon^2})}{\sqrt{|\nabla u^{\varepsilon}|^2+\varepsilon^2}}\right)\\
&\geq \eta''(H^{\varepsilon})H^{\varepsilon}\nabla H^{\varepsilon}\cdot \left(A^{\varepsilon}\frac{\nabla (\sqrt{|\nabla u^{\varepsilon}|^2+\varepsilon^2})}{\sqrt{|\nabla u^{\varepsilon}|^2+\varepsilon^2}}\right)\\
&\geq -C_{\varepsilon,T}\eta''(H^{\varepsilon})|H^{\varepsilon}||\nabla H^{\varepsilon}|,
\end{aligned}
\end{equation}
where $C_{\varepsilon,T}>0$ is a constant. Therefore,
\begin{equation}
\frac{d}{dt}\int_{\Omega}\eta(H^{\varepsilon})\,dx\leq C_{\varepsilon,T}\int_{\Omega}\eta''(H^{\varepsilon})|H^{\varepsilon}||\nabla H^{\varepsilon}|\,dx.
\end{equation}
Integrating in $t$, we have
\begin{equation}
\left[\int_{\Omega}\eta(H^{\varepsilon})\,dx\right]_{t=t_1}^{t_2}\leq C_{\varepsilon,T}\int_{t_1}^{t_2}\int_{\Omega}\eta''(H^{\varepsilon})|H^{\varepsilon}||\nabla H^{\varepsilon}|\,dxdt
\end{equation}
for any $0\leq t_1<t_2<T$. For each $\delta>0$, choose $\eta=\eta_{\delta}$ such that
\[
\begin{cases}
\eta_{\delta}(s)\rightarrow |s|\quad\textrm{locally uniformly in}~{\mathbb R}~\textrm{as}~\delta\to 0,\\
\spt~\eta_{\delta}''\subseteq [-\delta,\delta],\quad
0\leq \eta_{\delta}''\leq \frac{C}{\delta},
\end{cases}
\]
where $C>0$ is a constant independent of $\delta>0$. Then
\begin{equation}
\begin{aligned}
\left[\int_{\Omega}\eta_{\delta}(H^{\varepsilon})\,dx\right]_{t=t_1}^{t_2}&\leq C_{\varepsilon,T}\int_{t_1}^{t_2}\int_{\{|H^{\varepsilon}|\leq\delta\}}\eta_{\delta}''(H^{\varepsilon})|H^{\varepsilon}||\nabla H^{\varepsilon}|\,dxdt\\
&\leq C_{\varepsilon,T}\int_{t_1}^{t_2}\int_{\{|H^{\varepsilon}|\leq \delta\}}|\nabla H^{\varepsilon}|\,dxdt.
\end{aligned}
\end{equation}
Letting $\delta\to 0$, we obtain
\[
\left[\int_{\Omega}|H^{\varepsilon}|\,dx\right]_{t=t_1}^{t_2}\leq 0.
\]
Since the above integral is bounded at $t=0$ due to Lemma \ref{initial datum}(iii), this concludes the proof.
\end{proof}

The estimate \eqref{L1 estimate} allows us to take limits of various geometric quantities as $\varepsilon \to 0$. As a consequence, one may obtain the following.

\begin{lemma} \label{limit}
We have the following convergence as $\varepsilon \to 0$ for each time $t\geq 0$:
\begin{itemize}
\item[(i)] $\sqrt{|\nabla u^{\varepsilon}|^2+\varepsilon^2}\weakstar |\nabla u|$ in $L^{\infty}(\Omega)$,
\item[(ii)] $\nu^{\varepsilon}\rightarrow \nu=\frac{\nabla u}{|\nabla u|}$ strongly in $L^2(\{|\nabla u|>0\})$.
\end{itemize}
\end{lemma}
\begin{proof} The proof is identical to \cite[Theorem 3.1-3.3]{Evans4}. The argument is based on the compensated compactness and the estimate \eqref{L1 estimate}.
\end{proof}
We also prove an $L^2$ estimate for the approximate mean curvature. 
\begin{lemma} \label{lem:approx L2 estimate}
For any $T>0$,
\begin{equation} \label{approx L2 estimate}
\sup_{0<\varepsilon<1}\int_0^T\int_{\Omega}|H^{\varepsilon}|^2\sqrt{|\nabla u^{\varepsilon}|^2+\varepsilon^2}\,dxdt<\infty.
\end{equation}
\end{lemma}

\begin{proof} Recalling that $\nu^{\varepsilon}\cdot \nu_{\partial \Omega}=0$ on $\partial \Omega$, we compute
\[
\begin{aligned}
\frac{d}{dt}\int_{\Omega}\sqrt{|\nabla u^{\varepsilon}|^2+\varepsilon^2}\,dx&=\int_{\Omega}\nu^{\varepsilon}\cdot \nabla (\partial_tu^{\varepsilon})\,dx=\int_{\partial \Omega}\partial_t u^{\varepsilon}(\nu^{\varepsilon}\cdot \nu_{\partial \Omega})\,d\mathcal{H}^{N-1}-\int_{\Omega}H^{\varepsilon}\partial_tu^{\varepsilon}\,dx\\
&=-\int_{\Omega}|H^{\varepsilon}|^2\sqrt{|\nabla u^{\varepsilon}|^2+\varepsilon^2}\,dx.
\end{aligned}
\]
Integrating in $t$, we have
\begin{equation} \label{difference approx area}
\int_{\Omega}\sqrt{|\nabla u^\varepsilon(\cdot,T)|^2+\varepsilon^2}\,dx+
\int_0^T\int_{\Omega}|H^{\varepsilon}|^2\sqrt{|\nabla u^{\varepsilon}|^2+\varepsilon^2}\,dxdt
=\int_{\Omega}\sqrt{|\nabla g|^2+\varepsilon^2}\,dx.
\end{equation}
Since the right-hand side may be bounded independent of $\varepsilon\in(0,1)$, we conclude the proof. 
\end{proof}
The proof of the following claim is identical to \cite[Lemma 4.2]{Evans4}.
\begin{lemma} \label{def of curvature}
$\partial_t u=0$ a.e. on $\{|\nabla u|=0\}\subset\Omega\times(0,\infty)$
and
\begin{equation}
H^\varepsilon \sqrt{|\nabla u^\varepsilon|^2+\varepsilon^2}\weakstar H|\nabla u|\,\,\mbox{in}\,\,L^{\infty}(\Omega\times(0,\infty)).
\end{equation}
\end{lemma}

Then we claim the $L^2$ estimate on the mean curvature $H$. Theorem \ref{thm:L2 estimate} may be derived from Lemma \ref{lem:approx L2 estimate}. See \cite[Lemma 4.3]{Evans4} for the details.

\begin{theorem} \label{thm:L2 estimate}
We have
\begin{equation}
\int_0^\infty dt\int_{\Omega}H^2|\nabla u|\,dx\leq \int_{\Omega}|\nabla g|\,dx<\infty.
\end{equation}
\end{theorem}

We next establish the first variation formula and Brakke's inequality of the level set solution. The proofs are similar to \cite{Evans4}, hence we simply sketch their proofs. To prove the first variation formula, we deduce the $\varepsilon$-version of the first variation formula for $u^{\varepsilon}$.

\begin{lemma} \label{lem:approx 1st var fml}
For any $X\in \rm{Lip}(\overline{\Omega};{\mathbb R}^N)$,
\begin{equation} \label{approx 1st var fml}
\begin{aligned}
&\int_{\Omega}\tr\left((I-\nu^{\varepsilon}\otimes \nu^{\varepsilon})\nabla X\right)\sqrt{|\nabla u^{\varepsilon}|^2+\varepsilon^2}\,dx\\
&=\int_{\Omega}H^{\varepsilon}(\nu^{\varepsilon}\cdot X)\sqrt{|\nabla u^{\varepsilon}|^2+\varepsilon^2}\,dx+\int_{\partial \Omega}(X\cdot \nu_{\partial \Omega})\sqrt{|\nabla u^{\varepsilon}|^2+\varepsilon^2}\,d\mathcal{H}^{N-1}.
\end{aligned}
\end{equation}
\end{lemma}

\begin{proof}
Recalling that $\nu^{\varepsilon}\cdot \nu_{\partial \Omega}=0$, we compute
\[
\begin{aligned}
&\int_{\Omega}H^{\varepsilon}(\nu^{\varepsilon}\cdot X)\sqrt{|\nabla u^{\varepsilon}|^2+\varepsilon^2}\,dx=\int_{\Omega}H^{\varepsilon}(\nabla u^{\varepsilon}\cdot X)\,dx\\
&=\int_{\partial \Omega}(\nu^{\varepsilon}\cdot \nu_{\partial \Omega})(\nabla u^{\varepsilon}\cdot X)\,d\mathcal{H}^{N-1}-\int_{\Omega}{}^t\nu^{\varepsilon}\nabla^2u^{\varepsilon}X+{}^t\nu^{\varepsilon}\nabla X\nabla u^{\varepsilon}\,dx\\
&=-\int_{\Omega}\nabla \left(\sqrt{|\nabla u^{\varepsilon}|^2+\varepsilon^2}\right)\cdot X+({}^t\nu^{\varepsilon}\nabla X\nu^{\varepsilon})\sqrt{|\nabla u^{\varepsilon}|^2+\varepsilon^2}\,dx\\
&=-\int_{\partial \Omega}(X\cdot \nu_{\partial \Omega})\sqrt{|\nabla u^{\varepsilon}|^2+\varepsilon^2}\,d\mathcal{H}^{N-1}+\int_{\Omega}(\textrm{div}~X-{}^t\nu^{\varepsilon}\nabla X\nu^{\varepsilon})\sqrt{|\nabla u^{\varepsilon}|^2+\varepsilon^2}\,dx\\
&=-\int_{\partial \Omega}(X\cdot \nu_{\partial \Omega})\sqrt{|\nabla u^{\varepsilon}|^2+\varepsilon^2}\,d\mathcal{H}^{N-1}+\int_{\Omega}\tr\left((I-\nu^{\varepsilon}\otimes \nu^{\varepsilon})\nabla X\right)\sqrt{|\nabla u^{\varepsilon}|^2+\varepsilon^2}\,dx.
\end{aligned}
\]
\end{proof}

We take the limit of the formula \eqref{approx 1st var fml} to obtain the following, which may be regarded as the first variation formula of the level set solution $u$.

\begin{theorem} \label{thm:1st var fml}
For a.e. $t\geq 0$, we have the following equality for each $X\in \rm{Lip}(\overline{\Omega};{\mathbb R}^N)$ with $X\cdot \nu_{\partial \Omega}=0$ on $\partial \Omega$:
\begin{equation} \label{1st var fml}
\int_{\Omega}\tr\left((I-\nu\otimes \nu)\nabla X\right)|\nabla u|\,dx=\int_{\Omega}H(\nu\cdot X)|\nabla u|\,dx.
\end{equation}
\end{theorem}

\begin{proof}
By Lemma \ref{lem:approx 1st var fml}, we have
\[
\int_{\Omega}\tr\left((I-\nu^{\varepsilon}\otimes \nu^{\varepsilon})\nabla X\right)\sqrt{|\nabla u^{\varepsilon}|^2+\varepsilon^2}\,dx=\int_{\Omega}H^{\varepsilon}(\nu^{\varepsilon}\cdot X)\sqrt{|\nabla u^{\varepsilon}|^2+\varepsilon^2}\,dx.
\]

Hence for any $T>0$ and any $\phi\in L^{\infty}(0,T)$,
\begin{equation}
\int_0^T\phi\int_{\Omega}\tr\left((I-\nu^{\varepsilon}\otimes \nu^{\varepsilon})\nabla X\right)\sqrt{|\nabla u^{\varepsilon}|^2+\varepsilon^2}\,dxdt=\int_0^T \phi\int_{\Omega}H^{\varepsilon}(\nu^{\varepsilon}\cdot X)\sqrt{|\nabla u^{\varepsilon}|^2+\varepsilon^2}\,dxdt.
\end{equation}
Now separate the integration into $\{|\nabla u|>0\}$ and $\{|\nabla u|=0\}$, and take the limit of each part as $\varepsilon\to 0$ using Lemma \ref{limit} and \ref{def of curvature} to obtain
\begin{equation}
\int_0^T\phi\int_{\Omega}\tr\left((I-\nu\otimes \nu)\nabla X\right)|\nabla u|\,dxdt=\int_0^T \phi\int_{\Omega}H(\nu\cdot X)|\nabla u|\,dxdt.
\end{equation}
See \cite[Theorem 5.1]{Evans4} for the details.
\end{proof}

In a similar manner, one may prove the following Brakke's inequality.

\begin{theorem} \label{thm:Brakke ineq}
For any $0\leq t_1<t_2<\infty$ and any $\phi\in C^1(\overline{\Omega}\times [0,\infty))$ with $\phi\geq 0$,
\begin{equation} \label{Brakke ineq}
\left[\int_{\Omega}\phi|\nabla u|\,dx\right]_{t=t_1}^{t_2}\leq \int_{t_1}^{t_2}\int_{\Omega}\left(\partial_t\phi-H(\nu\cdot \nabla \phi)-\phi H^2\right)|\nabla u|\,dxdt.
\end{equation}
\end{theorem}

\begin{proof} Recalling that $\nu^{\varepsilon}\cdot \nu_{\partial \Omega}=0$ on $\partial \Omega$, we compute
\[
\begin{aligned}
&\frac{d}{dt}\int_{\Omega}\phi\sqrt{|\nabla u^{\varepsilon}|^2+\varepsilon^2}\,dx=\int_{\Omega}\partial_t \phi\sqrt{|\nabla u^{\varepsilon}|^2+\varepsilon^2}+\phi(\nu^{\varepsilon}\cdot \nabla(\partial_t u^{\varepsilon}))\,dx\\
&=\int_{\Omega}\partial_t \phi\sqrt{|\nabla u^{\varepsilon}|^2+\varepsilon^2}-\partial_t u^{\varepsilon}(\nu^{\varepsilon}\cdot \nabla \phi)-\phi \partial_t u^{\varepsilon}H^{\varepsilon}\,dx+\int_{\partial \Omega}\phi \partial_t u^{\varepsilon}(\nu^{\varepsilon}\cdot \nu_{\partial \Omega})\,d\mathcal{H}^{N-1}\\
&=\int_{\Omega}\partial_t \phi\sqrt{|\nabla u^{\varepsilon}|^2+\varepsilon^2}-H^{\varepsilon}(\nu^{\varepsilon}\cdot \nabla \phi)\sqrt{|\nabla u^{\varepsilon}|^2+\varepsilon^2}-\phi |H^{\varepsilon}|^2\sqrt{|\nabla u^{\varepsilon}|^2+\varepsilon^2}\,dx.
\end{aligned}
\]
Integrating in $t$, we have
\begin{equation}
\begin{aligned}
&\left[\int_{\Omega}\phi\sqrt{|\nabla u^{\varepsilon}|^2+\varepsilon^2}\,dx\right]_{t=t_1}^{t_2}\\
&=\int_{t_1}^{t_2}\int_{\Omega}\partial_t \phi\sqrt{|\nabla u^{\varepsilon}|^2+\varepsilon^2}-H^{\varepsilon}(\nu^{\varepsilon}\cdot \nabla \phi)\sqrt{|\nabla u^{\varepsilon}|^2+\varepsilon^2}-\phi |H^{\varepsilon}|^2\sqrt{|\nabla u^{\varepsilon}|^2+\varepsilon^2}\,dxdt.
\end{aligned}
\end{equation}
By the same argument as in \cite[Theorem 5.2]{Evans4}, one may obtain \eqref{Brakke ineq}.
\end{proof}

\section{Interpretation to varifolds}
In Section 4, we interpret formulae in Section 3 in the language of varifolds. We shall prove that almost all level sets of $u$ are unit-density Brakke flows with appropriate boundary conditions. First of all, we collect basic properties of $\Gamma_t^{\gamma}=\{x\in \overline{\Omega}~|~u(x,t)=\gamma\}$.

\begin{lemma} \label{basic}
For a.e. $(\gamma,t)\in \mathbb R\times(0,\infty)$, we have the following properties. 
\begin{itemize}
\item[(i)] For $\mathcal{H}^{N-1}$-a.e. $x\in \Gamma_t^{\gamma}$, $u$ is differentiable at $(x,t)$ and $|\nabla u(x,t)|\not=0$.
\item [(ii)] $\mathcal{H}^{N-1}(\Gamma_t^{\gamma})\leq \mathcal{H}^{N-1}(\Gamma_0^{\gamma})<\infty$ .
\item[(iii)] $\mathcal{H}^{N-1}(\Gamma_t^{\gamma}\cap \partial \Omega)=0$.
\item[(iv)] $\Gamma_t^{\gamma}$ is $(N-1)$-rectifiable. Moreover, its tangent space is given by
\begin{equation}
T_x\Gamma_t^{\gamma}=I-\nu\otimes \nu
\quad\mathcal{H}^{N-1}\textrm{-a.e. on}~\Gamma_t^{\gamma}
\end{equation}
\end{itemize}
\end{lemma}

\begin{proof} The properties (i) and (iii) follow from the coarea formula (see \cite[Lemma 6.1]{Evans4} for 
the detail)
and (iv) follows from the differentiability of $u$ and the well-known characterization of rectifiability. 
We only verify (ii). From \eqref{difference approx area}, we see
\[
\int_{\Omega}\sqrt{|\nabla u^{\varepsilon}|^2+\varepsilon^2}\,dx\leq \int_{\Omega}\sqrt{|\nabla g|^2+\varepsilon^2}\,dx.
\]
Letting $\varepsilon\to 0$, we have
\begin{equation} \label{area decreasing}
\int_{\Omega}|\nabla u|\,dx\leq \int_{\Omega}|\nabla g|\,dx.
\end{equation}
Take any $\eta\in C^{\infty}({\mathbb R})$ with $\eta'>0$. Then $\tilde{u}:=\eta(u)$ is a unique viscosity solution to \eqref{level set MCF} with an initial data $\tilde{g}:=\eta(g)$. From \eqref{area decreasing} for $\tilde{u}$ and $\tilde{g}$, we see
\[
\int_{\Omega}\eta'(u)|\nabla u|\,dx=\int_{\Omega}|\nabla \tilde{u}|\,dx\leq \int_{\Omega}|\nabla \tilde{g}|\,dx=\int_{\Omega}\eta'(g)|\nabla g|\,dx.
\]
For any $T>0$ and any $\phi\in L^{\infty}(0,T)$ with $\phi\geq 0$,
\[
\int_0^T \phi(t)\int_{\Omega}\eta'(u)|\nabla u|\,dxdt\leq \int_0^T\phi(t)\int_{\Omega}\eta'(g)|\nabla g|\,dxdt.
\]
Using the coarea formula, we have
\begin{equation}
\int_0^T\int_{{\mathbb R}}\phi(t)\eta'(\gamma)\mathcal{H}^{N-1}(\Gamma_t^{\gamma})\,d\gamma dt\leq \int_0^T\int_{{\mathbb R}}\phi(t)\eta'(\gamma)\mathcal{H}^{N-1}(\Gamma_0^{\gamma})\,d\gamma dt.
\end{equation}
By the arbitrariness of $\eta$ and $\phi$, we conclude $\mathcal{H}^{N-1}(\Gamma_t^{\gamma})\leq \mathcal{H}^{N-1}(\Gamma_0^{\gamma})$  for a.e. $(\gamma,t)$. On the other hand,
\begin{equation}
\int_{{\mathbb R}}\mathcal{H}^{N-1}(\Gamma_0^{\gamma})\,d\gamma=\int_{\Omega}|\nabla g|\,dx<\infty
\end{equation}
implies $\mathcal{H}^{N-1}(\Gamma_0^{\gamma})<\infty$ for a.e. $\gamma$.
\end{proof}

For a.e. $(\gamma,t)\in\mathbb R\times(0,\infty)$, let $V_t^{\gamma}$ be the
unit-density varifold induced by $\Gamma_t^{\gamma}$.
By the definition, its weight measure $\|V_t^{\gamma}\|$ is $\mathcal{H}^{N-1}\lfloor_{\Gamma_t^{\gamma}}$ and by Lemma \ref{basic}(iv), its first variation is given by
\begin{equation}\label{dV}
\delta V_t^{\gamma}(X)=\int_{\Gamma_t^{\gamma}}\tr((I-\nu\otimes \nu)\nabla X)\,d\mathcal{H}^{N-1}
\quad\textrm{for}~X\in C^1(\overline{\Omega};{\mathbb R}^N).
\end{equation}
The coarea formula and Theorem \ref{thm:L2 estimate} show the following. 

\begin{lemma} \label{lem:L2 estimate var}
For a.e. $\gamma\in {\mathbb R}$, we have
\begin{equation} \label{L2 estimate var}
\int_0^\infty dt\int_{\Omega}H^2\,d\|V_t^{\gamma}\|<\infty. 
\end{equation}
\end{lemma}

Next we prove Lemma \ref{main:total var fin}. 

\begin{proof}[Proof of Lemma \ref{main:total var fin}]
Take any $\phi\in L^{\infty}(0,T),~\eta\in \textrm{Lip}({\mathbb R})$ with $\phi\geq 0,~\eta\geq 0$. From Lemma \ref{lem:approx 1st var fml}, we have for any $X\in \textrm{Lip}(\overline{\Omega};{\mathbb R}^N)$,
\begin{equation}
\begin{aligned}
&\lim_{\varepsilon\to 0}\int_0^T \phi(t)\int_{\partial \Omega}(X\cdot \nu_{\partial \Omega})\sqrt{|\nabla u^{\varepsilon}|^2+\varepsilon^2}\,d\mathcal{H}^{N-1}dt\\
&=\int_0^T \phi(t)\int_{\Omega}\left\{\tr\left((I-\nu\otimes \nu)\nabla X\right)-H(\nu\cdot X)\right\}|\nabla u|\,dxdt
\end{aligned}
\end{equation}
by the same argument as in Theorem \ref{thm:1st var fml}. We replace $X$ by $\eta(u)X$ and apply the coarea formula to obtain
\begin{equation} \label{1st var with bdry}
\begin{aligned}
&\lim_{\varepsilon\to 0}\int_0^T \phi(t)\int_{\partial \Omega}\eta(u)(X\cdot \nu_{\partial \Omega})\sqrt{|\nabla u^{\varepsilon}|^2+\varepsilon^2}\,d\mathcal{H}^{N-1}dt\\
&=\int_0^T \phi(t)\int_{\Omega}\eta(u)\left\{\tr\left((I-\nu\otimes \nu)\nabla X\right)-H(\nu\cdot X)\right\}|\nabla u|\,dxdt\\
&=\int_0^T\int_{{\mathbb R}}\phi(t)\eta(\gamma)\int_{\Gamma_t^{\gamma}}\tr\left((I-\nu\otimes \nu)\nabla X\right)-H(\nu\cdot X)\,d\mathcal{H}^{N-1}d\gamma dt.
\end{aligned}
\end{equation}
Now choose a vector field $X_0\in C^{\infty}(\overline{\Omega})$ such that $X_0=\nu_{\partial \Omega}$ on $\partial \Omega$. By \eqref{1st var with bdry},
\begin{equation}\label{bdry 1st var estimate}
\begin{aligned}
&\lim_{\varepsilon\to 0}\int_0^T \phi(t)\int_{\partial \Omega}\eta(u)\sqrt{|\nabla u^{\varepsilon}|^2+\varepsilon^2}\,d\mathcal{H}^{N-1}dt\\
&\leq C\int_0^T\int_{{\mathbb R}}\phi(t)\eta(\gamma)\left\{\mathcal{H}^{N-1}(\Gamma_t^{\gamma})+\|H\|_{L^1(\Gamma_t^{\gamma})}\right\}d\gamma dt,
\end{aligned}
\end{equation}
where $C=C(\|X_0\|_{C^1},N)>0$ is a constant. Combining \eqref{1st var with bdry} and \eqref{bdry 1st var estimate}, we have
\begin{equation}
\begin{aligned}
&\int_0^T\int_{{\mathbb R}}\phi(t)\eta(\gamma)\int_{\Gamma_t^{\gamma}}\tr\left((I-\nu\otimes \nu)\nabla X\right)\,d\mathcal{H}^{N-1}d\gamma dt\\
&\leq (C+1)\|X\|_{\infty}\int_0^T \int_{{\mathbb R}}\phi(t)\eta(\gamma)\left\{\mathcal{H}^{N-1}(\Gamma_t^{\gamma})+\|H\|_{L^1(\Gamma_t^{\gamma})}\right\}d\gamma dt.
\end{aligned}
\end{equation}
By the arbitrariness of $\phi$ and $\eta$, we have for a.e $(\gamma,t)$,
\begin{equation}
\int_{\Gamma_t^{\gamma}}\tr\left((I-\nu\otimes \nu)\nabla X\right)\,d\mathcal{H}^{N-1}\leq (C+1)\|X\|_{\infty}\left(\mathcal{H}^{N-1}(\Gamma_t^{\gamma})+\|H\|_{L^1(\Gamma_t^{\gamma})}\right).
\end{equation}
Therefore,
\begin{equation} \label{total var finite}
\|\delta V_t^{\gamma}\|(\overline{\Omega})\leq (C+1)\left(\mathcal{H}^{N-1}(\Gamma_t^{\gamma})+\|H\|_{L^1(\Gamma_t^{\gamma})}\right).
\end{equation}
By Lemma \ref{basic}(ii) and Lemma \ref{lem:L2 estimate var}, we conclude the proof.
\end{proof}

By Lemma \ref{main:total var fin}, we may define the first variation on $\partial \Omega$ and its tangential component as in Definition \ref{1st var on bdry}. Now we prove Theorem \ref{main:boundary conditions}.

\begin{proof}[Proof of Theorem \ref{main:boundary conditions}]
The properties (i) and (ii) follow from Lemma \ref{lem:L2 estimate var}, \eqref{dV} and another application of the coarea formula to Theorem \ref{thm:1st var fml}. It remains to demonstrate (iii). Take any $X\in C^1(\partial \Omega;{\mathbb R}^N)$. Let $Y\in C^1(\overline{\Omega};{\mathbb R}^N)$ such that $Y=X-(X\cdot \nu_{\partial \Omega})\nu_{\partial \Omega}$ on $\partial \Omega$. We apply the coarea formula to Theorem \ref{thm:1st var fml} to obtain $\delta V_t^{\gamma}(Y)=\int_{\overline{\Omega}}H\nu\cdot Xd\|V_t^{\gamma}\|$ for a.e. $(\gamma,t)$. By (i) and Lemma \ref{basic}(iii), we deduce $\delta V_t^{\gamma}(Y)=\delta V_t^{\gamma}\lfloor_{\Omega}(Y)$. Therefore $\delta V_t^{\gamma}\lfloor_{\partial \Omega}^{\top}(X)=\delta V_t^{\gamma}\lfloor_{\partial \Omega}(Y)=0$. 
\end{proof}

Lastly, we show Theorem \ref{main:Brakke ineq}.

\begin{proof}[Proof of Theorem \ref{main:Brakke ineq}]
Take any $\eta\in C^{\infty}({\mathbb R})$ with $\eta'>0$. Then $\tilde{u}:=\eta(u)$ is a unique viscosity solution to \eqref{level set MCF} with an initial data $\tilde{g}:=\eta(g)$. By Theorem \ref{thm:Brakke ineq} for $\tilde{u}$, we obtain
\begin{equation}
\begin{aligned}
\left[\int_{\Omega}\phi\eta'(u)|\nabla u|\,dx\right]_{t=t_1}^{t_2}\leq \int_{t_1}^{t_2}\int_{\Omega}\left(\partial_t\phi-H(\nu\cdot \nabla \phi)-\phi H^2\right)\eta'(u)|\nabla u|\,dxdt.
\end{aligned}
\end{equation}
Using the coarea formula, we have
\begin{equation}
\int_{{\mathbb R}}\eta'\,d\gamma\left[\int_{\Gamma_t^{\gamma}}\phi\,d\mathcal{H}^{N-1}\right]_{t=t_1}^{t_2}\leq \int_{{\mathbb R}}\eta'\,d\gamma\int_{t_1}^{t_2}\int_{\Gamma_t^{\gamma}}\left(\partial_t\phi-H(\nu\cdot \nabla \phi)-\phi H^2\right)\,d\mathcal{H}^{N-1}dt.
\end{equation}
Since $\eta$ is arbitrary, we conclude the proof.
\end{proof}

As stated in Section 2, for a.e. $\gamma$, one may verify that $V_t^{\gamma}$ has the semi-decreasing property, i.e. the mapping $t\mapsto \|V_t^{\gamma}\|(\phi)-C_{\phi}t$ is non-increasing for non-negative $\phi\in C^2(\overline{\Omega})$ with $C_{\phi}:=\mathcal{H}^{N-1}(\Gamma_0^{\gamma})\sup_{\{\phi>0\}}\frac{|\nabla \phi|^2}{\phi}$. Thus we may define $\|V_t^{\gamma}\|(\phi)$ left-continuously for all $t\in [0,\infty)$. Moreover, by Lemma \ref{basic}(ii), $\|V_t^{\gamma}\|(\phi)\leq \|\phi\|_{\infty}\mathcal{H}^{N-1}(\Gamma_0^{\gamma})$. Therefore, we may define a Radon measure $\|V_t^{\gamma}\|$ for all $t\in [0,\infty)$. Since $t\mapsto \|V_t^{\gamma}\|$ is left-continuous, Theorem \ref{main:Brakke ineq} is still valid for any $0\leq t_1<t_2<\infty$.

\section*{Appendix : Proof of Lemma \ref{initial datum}}
In this appendix, we carry out a technical construction of the appropriate initial data $g$ as in Lemma \ref{initial datum}. For the simplicity of notation, we will define the initial data $g$ on the whole space ${\mathbb R}^N$.

\begin{proof}[Proof of Lemma \ref{initial datum}]
Step1. For each $A\subseteq {\mathbb R}^N$, we define the signed distance by $d_A(x):=\dist(x,A)-\dist(x,A^c)$ for $x\in {\mathbb R}^N$. Since $U,\Omega$ are smooth, there exists $R>0$ such that $d_U$ is smooth in $\{-2R<d_U<2R\}$ and that $d_{\Omega}$ is smooth in $\{-2R<d_{\Omega}<2R\}$. Choose $\eta\in C_c^{\infty}({\mathbb R})$ to satisfy
\[
\eta(s)=1~\textrm{if}~|s|<R,\quad
\eta(s)=0~\textrm{if}~|s|\geq \tfrac{3}{2}R,\quad
0\leq \eta\leq 1.
\]
Define $X\in C^{\infty}({\mathbb R}^N;{\mathbb R}^N)$ by
\begin{equation}
X:=\eta(d_U)\eta(d_{\Omega})\{(I-\nabla d_{\Omega}\otimes \nabla d_{\Omega})\nabla d_U+d_{\Omega}X_0\},
\end{equation}
where $X_0$ is a vector field to be defined later. Here, remark that $X|_{\partial \Omega}$ is a tangent vector field of $\partial \Omega$. Let $\Phi:{\mathbb R}^N\times {\mathbb R}\rightarrow {\mathbb R}^N$ be a 1-parameter group generated by $X$, i.e.
\begin{equation}
\begin{cases}
\frac{d}{ds}\Phi(x,s)=X(\Phi(x,s)),\\
\Phi(x,0)=x.
\end{cases}
\end{equation}
By Assumption \ref{asm}, we compute
\begin{equation}
\left.\frac{d}{ds}d_U(\Phi(x,s))\right|_{s=0}=1-(\nabla d_U\cdot \nabla d_{\Omega})^2=1\not=0
\quad\textrm{on}~\partial U\cap \partial \Omega.
\end{equation}
By the implicit function theorem, there exist a neighborhood $W$ of $\partial U\cap \partial \Omega$, a positive number $s_0>0$ and a smooth function $f:W\rightarrow {\mathbb R}$ such that
\begin{equation}
\{(x,s)\in W\times (-s_0,s_0)~|~f(x)=s\}=\{(x,s)\in W\times (-s_0,s_0)~|~d_U(\Phi(x,-s))=0\}.
\end{equation}

\begin{figure}[htbp]
\centering	
\includegraphics[height=4cm,bb=0 0 712 441]{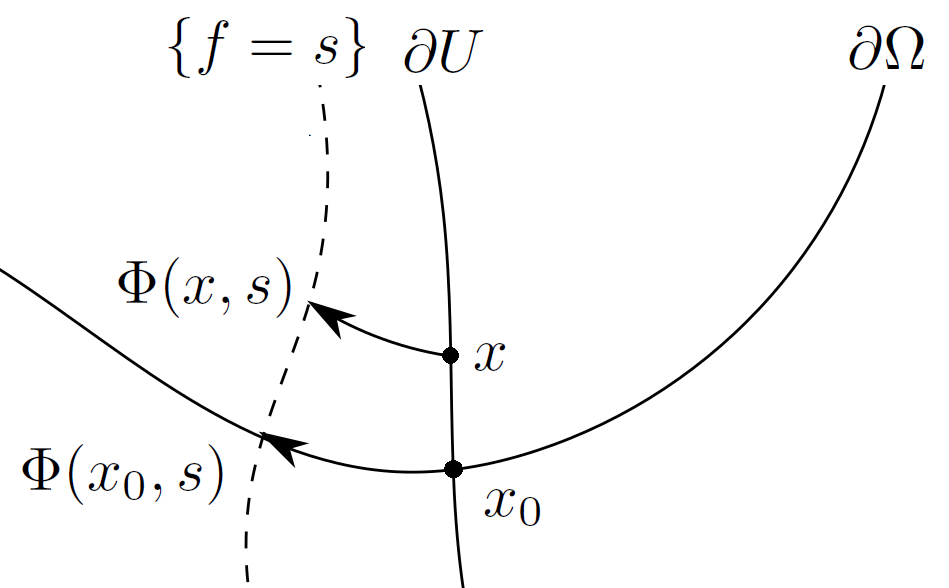}	
\caption{Definition of $f(x)$}\label{fig1}
\end{figure}

By differentiating $d_U(\Phi(x,-f(x)))=0$, we compute
\begin{equation} \label{derivative of f}
\nabla f(x)=\frac{\nabla \Phi(x,-f(x))\nabla d_U(\Phi(x,-f(x)))}{X(\Phi(x,-f(x)))\cdot \nabla d_U(\Phi(x,-f(x)))}.
\end{equation}
Moreover, we may assume the following conditions:
\begin{itemize}
\item[(F1)] $W\subseteq \{-R<d_U<R\}\cap \{-R<d_{\Omega}<R\}$.
\item[(F2)] $\nabla f\cdot \nabla d_U>0$ in $W$.
\item[(F3)] $d_U$ and $f$ have the same sign in $W$.
\item[(F4)] $\Phi(\partial U\cap \partial \Omega,(-s_0,s_0))=\partial \Omega\cap W$.
\item[(F5)] $\nabla f\cdot \nabla d_{\Omega}=0$ on $\partial \Omega\cap W$.
\end{itemize}
In the rest part of Step 1, we verify these properties (F1)-(F5). At first, it is easy to check the property (F1). From \eqref{derivative of f}, for $x_0\in \partial U\cap \partial \Omega$,
\[ \nabla f(x_0)\cdot \nabla d_U(x_0)=\frac{\nabla \Phi(x_0,0)\nabla d_U(x_0)}{X(x_0)\cdot \nabla d_U(x_0)}\cdot \nabla d_U(x_0)=\frac{\nabla d_U(x_0)}{\nabla d_U(x_0)\cdot \nabla d_U(x_0)}\cdot \nabla d_U(x_0)=1.\]
Therefore, taking $W$ smaller, we may assume (F2). Next we verify (F3). On a neighborhood of $(\partial U\cap \partial \Omega)\times \{0\}$, we compute
\begin{equation}\label{sign of dist}
\begin{aligned}
d_U(x)&=d_U(\Phi(\Phi(x,-f(x)),f(x))=\int_0^{f(x)}\frac{d}{ds}\left[d_U(\Phi(\Phi(x,-f(x)),s)\right]\,ds\\
&=\int_0^{f(x)}\nabla d_U(\Phi(\Phi(x,-f(x)),s))\cdot X(\Phi(\Phi(x,-f(x)),s))\,ds.
\end{aligned}
\end{equation}
Here we may assume the integrand of \eqref{sign of dist} is positive since
\[\nabla d_U(\Phi(\Phi(x,-f(x)),s))\cdot X(\Phi(\Phi(x,-f(x)),s))=\nabla d_U(x)\cdot \nabla d_U(x)=1\]
on $(\partial U\cap \partial \Omega)\times \{0\}$. Thus, the property (F3) holds true. Moreover, choosing $s_0$ smaller, we may assume $\Phi(\partial U\cap \partial \Omega,(-s_0,s_0))\subseteq W$. Recalling that $X$ is a tangent vector field of $\partial \Omega$, it is easy to see
\[\Phi(\partial U\cap \partial \Omega,(-s_0,s_0))=\partial \Omega\cap W\cap f^{-1}((-s_0,s_0)).\]
Thus we obtain (F4) by replacing $W$ with $W\cap f^{-1}((-s_0,s_0))$. Finally, we verify (F5).
From (F4), it suffices to prove that for every $x_0\in \partial U\cap \partial \Omega$,
\begin{equation}
\alpha(s):=\nabla f(\Phi(x_0,s))\cdot \nabla d_{\Omega}(\Phi(x_0,s))
\end{equation}
is constantly 0 on $(-s_0,s_0)$. By the group property of $\Phi$, we have $\nabla \Phi(\Phi(x,s),-s)=\{\nabla \Phi(x,s)\}^{-1}$. Recalling \eqref{derivative of f} and {$|\nabla d_U|=1$}, we compute
\begin{equation}
\begin{aligned}
\alpha(s)&=\frac{\left(\{\nabla \Phi(x_0,s)\}^{-1}\nabla d_U(x_0)\right)\cdot \nabla d_{\Omega}(\Phi(x_0,s))}{X(x_0)\cdot \nabla d_U(x_0)}\\
&=\frac{\left(\{\nabla \Phi(x_0,s)\}^{-1}\nabla d_U(x_0)\right)\cdot \nabla d_{\Omega}(\Phi(x_0,s))}{\nabla d_U(x_0)\cdot \nabla d_U(x_0)}\\
&=\left(\{\nabla \Phi(x_0,s)\}^{-1}\nabla d_U(x_0)\right)\cdot \nabla d_{\Omega}(\Phi(x_0,s)).
\end{aligned}
\end{equation}
In particular, $\alpha(0)=(\{\nabla \Phi(x_0,0)\}^{-1}\nabla d_U(x_0))\cdot \nabla d_{\Omega}(x_0)=\nabla d_U(x_0)\cdot \nabla d_{\Omega}(x_0)=0$ since $\nabla \Phi(x_0,0)=\textrm{id}_{{\mathbb R}^N}$. Additionally, we calculate
\begin{equation}
\begin{aligned}
\frac{d}{ds}\{\nabla \Phi\}^{-1}&=-\{\nabla \Phi\}^{-1}\left\{\frac{d}{ds}(\nabla \Phi)\right\}\{\nabla \Phi\}^{-1}\\
&=-\{\nabla \Phi\}^{-1}\left\{\nabla(X\circ \Phi)\right\}\{\nabla \Phi\}^{-1}=-(\nabla X\circ\Phi)\{\nabla \Phi\}^{-1}.
\end{aligned}
\end{equation}
Therefore,
\begin{equation}
\begin{aligned}
\alpha'(s)&=\frac{d}{ds}\left[\left(\{\nabla \Phi(x_0,s)\}^{-1}\nabla d_U(x_0)\right)\cdot \nabla d_{\Omega}(\Phi(x_0,s))\right]\\
&=\left(\{\nabla \Phi(x_0,s)\}^{-1}\nabla d_U(x_0)\right)\cdot \left(\nabla^2 d_{\Omega}(\Phi(x_0,s))X(\Phi(x_0,s))\right)\\
&\quad -\left(\nabla X(\Phi(x_0,s))\{\nabla \Phi(x_0,s)\}^{-1}\nabla d_U(x_0)\right)\cdot \nabla d_{\Omega}(\Phi(x_0,s))\\
&=\left(\{\nabla \Phi(x_0,s)\}^{-1}\nabla d_U(x_0)\right)\cdot \left(\nabla^2 d_{\Omega}(y)X(y)-(\nabla X(y))^T\nabla d_{\Omega}(y)\right),
\end{aligned}
\end{equation}
here we write $y:=\Phi(x_0,s)\in \partial \Omega\cap W$. On the other hand, we compute 
\[
\begin{aligned}
(\nabla X)^T&=\left[\nabla\{(I-\nabla d_{\Omega}\otimes \nabla d_{\Omega})\nabla d_U+d_{\Omega}X_0\}\right]^T\\
&=-(\nabla d_{\Omega}\cdot \nabla d_U)\nabla^2d_{\Omega}-(\nabla d_{\Omega}\otimes \nabla d_U)\nabla^2 d_{\Omega}+(I-\nabla d_{\Omega}\otimes \nabla d_{\Omega})\nabla^2d_U\\
&\quad +X_0\otimes \nabla d_{\Omega}+d_{\Omega}(\nabla X_0)^T
\end{aligned}
\]
in $W$. Recalling that $d_{\Omega}=0$ on $\partial \Omega$, $|\nabla d_{\Omega}|=1$ and $(\nabla^2d_{\Omega})\nabla d_{\Omega}=0$, we calculate
\[
\begin{aligned}
&(\nabla^2 d_{\Omega})X-(\nabla X)^T\nabla d_{\Omega}\\
&\quad =(\nabla^2 d_{\Omega})(I-\nabla d_{\Omega}\otimes \nabla d_{\Omega})\nabla d_U+\nabla^2 d_{\Omega}(d_{\Omega}X_0)+(\nabla d_{\Omega}\cdot \nabla d_U)(\nabla^2d_{\Omega})\nabla d_{\Omega}\\
&\quad \quad
+(\nabla d_{\Omega}\otimes \nabla d_U)(\nabla^2 d_{\Omega})\nabla d_{\Omega}-(I-\nabla d_{\Omega}\otimes \nabla d_{\Omega})(\nabla^2d_U)\nabla d_{\Omega}\\
&\quad \quad 
-(X_0\otimes \nabla d_{\Omega})\nabla d_{\Omega}-d_{\Omega}(\nabla X_0)^T\nabla d_{\Omega}\\
&\quad =(\nabla^2d_{\Omega})\nabla d_U-(I-\nabla d_{\Omega}\otimes \nabla d_{\Omega})(\nabla^2d_U)\nabla d_{\Omega}-X_0
\end{aligned}
\]
on $\partial \Omega\cap W$. Hence we put $X_0:=(\nabla^2 d_{\Omega})\nabla d_U-(I-\nabla d_{\Omega}\otimes \nabla d_{\Omega})(\nabla^2d_U)\nabla d_{\Omega}$ to obtain $(\nabla^2 d_{\Omega})X-(\nabla X)^T\nabla d_{\Omega}=0$ on $\partial \Omega\cap W$. Then $\alpha'(s)=0$ on $(-s_0,s_0)$.\par

Step 2. We shall extend $f$ to ${\mathbb R}^N$. Take $r>0$ so that
\begin{equation}
\{-2r<d_U<2r\}\cap \{-2r<d_{\Omega}<2r\}\subseteq W.
\end{equation}
Choose $\zeta\in C_c^{\infty}({\mathbb R})$ satisfying
\[
\zeta(s)=1~\textrm{if}~|s|<r,\quad
\zeta(s)=0~\textrm{if}~|s|\geq \tfrac{3}{2}r,\quad
0\leq \zeta\leq 1.
\]
Then we define a smooth function $g_0$ in $\{-2r<d_U<2r\}$ by
\begin{equation}
g_0:=(1-\zeta\circ d_{\Omega})d_U+(\zeta\circ d_{\Omega})f.
\end{equation}
From (F1)-(F5), one may verify the following:
\begin{itemize}
\item[(G1)] There exists $\delta>0$ such that $V:=\{-\delta\leq g_0\leq \delta\}\subseteq \{-r<d_U<r\}$.
\item[(G2)] For $x\in \{-2r<d_U<2r\}$, $g_0(x)=0$ if and only if $d_U(x)=0$.
\item[(G3)] $g_0$ and $d_U$ have the same sign in $\{-2r<d_U<2r\}$.
\item[(G4)] $\nabla g_0\cdot \nabla d_{\Omega}=0$ on $\partial \Omega\cap \{-2r<d_U<2r\}$.
\item[(G5)] Choosing $\delta>0$ even smaller, $\nabla g_0\not=0$ on $V$.
\end{itemize}
Now take any $\phi\in C^{\infty}({\mathbb R})$ such that
\[
\begin{cases}
\phi(-s)=-\phi(s)&\textrm{for any}~s\in {\mathbb R},\\
\phi'(s)>0,~\phi''(s)\geq 0&\textrm{if}~-\delta<s\leq 0,\\
\phi(s)=-1&\textrm{if}~s\leq -\delta.
\end{cases}
\]
Then we set $g:{\mathbb R}^N\rightarrow {\mathbb R}$ by
\begin{equation}
g(x):=
\begin{cases}
\phi(g_0(x))&\textrm{if}~-2r<d_U(x)<2r,\\
1&\textrm{if}~d_U(x)>r,\\
-1&\textrm{if}~d_U(x)<-r.
\end{cases}
\end{equation}

Step 3. From (G1)-(G4), we see that $g\in C^{\infty}({\mathbb R}^N)$ with $\spt~\nabla g\subseteq V$ and the properties (i) and (ii). It remains to verify the property (iii). By (G5), we have $c_0:=\inf_{V}|\nabla g_0|>0$. We compute
\[
\begin{aligned}
\textrm{div}\left(\frac{\nabla g}{\sqrt{|\nabla g|^2+\varepsilon^2}}\right)&=\frac{1}{\sqrt{|\nabla g|^2+\varepsilon^2}}\,\tr\left(\left(I-\frac{\nabla g\otimes \nabla g}{|\nabla g|^2+\varepsilon^2}\right)\nabla^2 g\right)\\
&=\frac{1}{\sqrt{(\phi')^2|\nabla g_0|^2+\varepsilon^2}}\,\tr\left(\left(I-\frac{(\phi')^2\nabla g_0\otimes \nabla g_0}{(\phi')^2|\nabla g_0|^2+\varepsilon^2}\right)\phi'\nabla^2 g_0\right)\\
&\quad
+\frac{1}{\sqrt{(\phi')^2|\nabla g_0|^2+\varepsilon^2}}\,\tr\left(\left(I-\frac{(\phi')^2\nabla g_0\otimes \nabla g_0}{(\phi')^2|\nabla g_0|^2+\varepsilon^2}\right)(\phi''\nabla g_0\otimes \nabla g_0)\right).
\end{aligned}
\]
We estimate the first term as follows.
\[
\left|\frac{1}{\sqrt{(\phi')^2|\nabla g_0|^2+\varepsilon^2}}\,\tr\left(\left(I-\frac{(\phi')^2\nabla g_0\otimes \nabla g_0}{(\phi')^2|\nabla g_0|^2+\varepsilon^2}\right)\phi'\nabla^2 g_0\right)\right|\leq \frac{\phi'|\nabla^2 g_0|}{\sqrt{(\phi')^2|\nabla g_0|^2+\varepsilon^2}}\leq \frac{1}{c_0}|\nabla^2 g_0|.
\]
For the second term,
\[
\begin{aligned}
\left|\frac{1}{\sqrt{(\phi')^2|\nabla g_0|^2+\varepsilon^2}}\,\tr\left(\left(I-\frac{(\phi')^2\nabla g_0\otimes \nabla g_0}{(\phi')^2|\nabla g_0|^2+\varepsilon^2}\right)(\phi''\nabla g_0\otimes \nabla g_0)\right)\right|\\
=\frac{\varepsilon^2|\phi''||\nabla g_0|^2}{\left((\phi')^2|\nabla g_0|^2+\varepsilon^2\right)^{3/2}}\leq \frac{\varepsilon^2c_1|\phi''|}{\left(c_0^2(\phi')^2+\varepsilon^2\right)^{3/2}}|\nabla g_0|,
\end{aligned}
\]
where $c_1:=\sup_V|\nabla g_0|<\infty$. Therefore, by the coarea formula,
\begin{equation}
\begin{aligned}
&\int_{{\mathbb R}^N}\left|\textrm{div}~\left(\frac{\nabla g}{\sqrt{|\nabla g|^2+\varepsilon^2}}\right)\right|\,dx\leq C\int_{V}|\nabla^2 g_0|+\frac{\varepsilon^2|\phi''|}{\left(c_0^2(\phi')^2+\varepsilon^2\right)^{3/2}}|\nabla g_0|\,dx\\
&\leq C'\left(1+\int_{-\delta}^0\frac{\varepsilon^2\phi''(s)}{\left(c_0^2(\phi'(s))^2+\varepsilon^2\right)^{3/2}}\,ds\right)=C'\left(1+\int_{-\delta}^0\frac{d}{ds}\left(\frac{\phi'(s)}{\sqrt{c_0^2(\phi'(s))^2+\varepsilon^2}}\right)\,ds\right)\\
&\leq C'\left(1+\frac{\phi'(0)}{\sqrt{c_0^2(\phi'(0))^2+\varepsilon^2}}\right)<\infty,
\end{aligned}
\end{equation}
where $C, C'>0$ are some constants independent of $\varepsilon$.
\end{proof}

\bibliographystyle{plain}
\bibliography{Level set MCF with Neumann bdry cond}

\end{document}